\documentclass[11pt]{amsart}
\usepackage[latin1]{inputenc}
\usepackage{amsmath}
\usepackage{amsfonts}
\usepackage{amssymb}
\usepackage{graphicx}
\usepackage{fourier}
\usepackage{hyperref}
\usepackage{enumerate}
\usepackage{esint}
\usepackage{bm}
\usepackage{ulem}
\usepackage{xcolor} 
\usepackage {ifthen}
\usepackage{verbatim}         

\newtheorem{theorem}{Theorem}[section]
\newtheorem{lemma}[theorem]{Lemma}

\theoremstyle{definition}

\newtheorem{corollary}[theorem]{Corollary}
\theoremstyle{remark}
\newtheorem{remark}[theorem]{Remark}

\newcommand{\R}{\mathbb{R}}
\newcommand{\CC}{\mathbb{C}}

\newcommand{\N}{\mathbb{N}}

\renewcommand{\H}{\mathcal{H}}
\newcommand{\AAA}{\mathcal{A}}
\newcommand{\I}{\mathcal{I}}
\newcommand{\F}{\mathcal{F}}

\numberwithin{equation}{section}

\newcommand{\eps}{\varepsilon}

\definecolor{bostonuniversityred}{rgb}{0.8, 0.0, 0.0}
\newcommand{\amit}{\textcolor{black}} 
\definecolor{byzantium}{rgb}{0.44, 0.16, 0.39}

\newcommand{\SM}{\mathcal{S}^{>0}}
\newcommand{\SMU}{\mathcal{S}^{\geq 0}}

\newcommand{\rank}{\operatorname{rank}}
\newcommand{\diag}{\operatorname{diag}}
\newcommand{\tr}{\operatorname{Tr}}
\renewcommand{\span}{\operatorname{span}}

\renewcommand{\Re}{\operatorname{Re}}

\renewcommand{\det}{\operatorname{det}}
\newcommand{\IK}{{\mathcal{I}(K)}}
\newcommand{\bea}{\textcolor{black}}
\title[Optimal non-symmetric Fokker-Planck equation]{Optimal non-symmetric Fokker-Planck equation for the convergence to a given equilibrium}
\author{Anton Arnold, Beatrice Signorello}
\address{Vienna University of Technology, Institute of Analysis and Scientific Computing, Wiedner Hauptstr. 8-10, A-1040 Wien, Austria}
\email{anton.arnold@tuwien.ac.at; beatrice.signorello@tuwien.ac.at}

\date{\today}
\begin{document}

\begin{abstract}
This paper is concerned with finding Fokker-Planck equations in $\R^d$ with the fastest exponential decay towards a given equilibrium. For a prescribed, anisotropic Gaussian we determine a non-symmetric Fokker-Planck equation with linear drift that shows the highest exponential decay rate for the convergence of its solutions towards equilibrium. At the same time it has to allow for a decay estimate with a multiplicative constant arbitrarily close to its infimum. 

Such an ``optimal'' Fokker-Planck equation is constructed explicitly with a diffusion matrix of rank one, hence being hypocoercive. 
In an $L^2$--analysis, we find that the maximum decay rate equals the maximum eigenvalue of the inverse covariance matrix, and that the infimum of the attainable multiplicative constant is 1, corresponding to the high-rotational limit in the Fokker-Planck drift. 
\bea{This analysis is complemented with numerical illustrations in 2D, and it includes a case study for time-dependent coefficient matrices.  
}
\end{abstract}

\maketitle

\noindent
{\tiny
{KEYWORDS:}
Fokker-Planck equation, fastest decay, non-symmetric perturbation, hypocoercivity}
\\{\tiny{MSC 2010:}
35Q84, 
35B40, 
35Q82, 
82C31 
}

\section{Introduction}\label{sec:intro}

The starting point of this paper is a linear, symmetric (or ``reversible'') Fokker-Planck (FP) equation on $\R^d$, \bea{$d\ge2$} with a corresponding, typically anisotropic Gaussian steady state. It is known from the literature (see \cite{LNP, GM}, e.g.) that the convergence to equilibrium can be accelerated by adding to the FP-equation non-symmetric perturbations that do not alter the equilibrium. It is hence a natural goal to find the ``optimal perturbation'' (in a sense to be made precise) such that the corresponding solutions converge the fastest to the fixed steady state. For FP-equations with \amit{fixed or variable diffusion matrices, this problem was studied, respectively,} in \cite{LNP} and \cite{GM}. A closely related question for the (kinetic) 1D Goldstein-Taylor system was recently studied in \cite{DE}: For a fixed (anti-symmetric) transport operator, the authors found the best (symmetric) relaxation operator, yielding the fastest exponential decay to equilibrium. 
\bea{For the same model, but with constant-in-$x$ relaxation, the propagator norm was previously computed in \cite{MM}.}\\

While we shall analyze this problem here in a pure PDE context, the origin of the question comes from a statistical and probabilistic setting: Let a given potential $V:\R^d \rightarrow \R$ satisfy $\int_{\R^d} e^{-V(x)} dx=1$, and define the probability density function
\begin{equation}
\label{eq:finfty}
f_{\infty, V}(x):=e^{-V(x)}.
\end{equation}
To compute expectations with respect to the associated probability measure $\mu_{V}$, e.g.\ via a Markov chain Monte Carlo algorithm (see \cite{D}), one needs to construct an ergodic Markov process $\left (X_t \right )_{t \geq 0}$ with the unique invariant law $\mu_{V}$, i.e.
\begin{equation}
\label{eq:Xtconvergence}
\operatorname{law}{(X_t)} \rightarrow \mu_{V}, \qquad \mbox{as }\: t \rightarrow \infty.
\end{equation}
The efficiency of such algorithms can be measured by the speed of convergence in \eqref{eq:Xtconvergence}. This motivates to pursue the following objective: find the fastest among all possible processes that sample from the same equilibrium $\mu_{V}$. 
A classical way to sample from $\mu_{V}$ is to consider 
a standard Brownian motion with drift $-\nabla V$.

The probability density function $f_t$ of the process $X_t$ at time $t$ then  solves the Fokker-Planck equation
\begin{equation}
\label{eq:FPsymmetric}
\partial_t f_t= \operatorname{div}_x( \nabla_x f_t +\nabla_x V f_t )=:-Lf_t, \quad t>0.
\end{equation}
It is symmetric in the sense that its generator $L$ is symmetric in the Hilbert space 
$$\mathcal{H}:=L^2(\R^d, f_{\infty, V}^{-1})= \Big\{f:\R^d \rightarrow \R, \ \text{measurable s.t.} \int_{\R^d} f(x)^2 f^{-1}_{\infty,V}(x) dx < \infty \Big\}.
$$ 
Under appropriate assumptions on the potential $V$ (e.g. if $\frac{1}{2}|\nabla V(x)|^2-\Delta V(x) \rightarrow + \infty$ for $|x| \rightarrow \infty$, see \cite{Vi09}, A.19) it is possible to show that $f_t$ converges to the unique equilibrium $f_{\infty,V}$. Moreover, $L$ is coercive in $V_0^{\perp}$ with $V_0:=\mathrm{span}_{\R}{\{f_{\infty,V}\}} \subset \mathcal{H}$, i.e. $\exists \lambda>0$ such that
\begin{equation} 
\label{eq:coervlam}
 \langle L f, f \rangle_{\H}  \geq \lambda ||f||_{\H}^2, \qquad f \in V_0^{\perp}.
\end{equation} 
We shall assume in the sequel that this $\lambda$ is chosen as large as possible, i.e.\ as the spectral gap of $L$.
As a consequence, if $f_t$ is a solution of \eqref{eq:FPsymmetric}, then 
\begin{equation}
\label{eq:coercRevFP}
||f_t-f_{\infty,V}||_{\mathcal{H}}\leq e^{-\lambda t} ||f_0-f_{\infty,V}||_{\mathcal{H}}, 
\end{equation}
for any normalized initial condition $f_0 \in \mathcal{H}$ \amit{(see Proposition 9 in \cite{Vi09}).}
So we have a purely exponential convergence estimate. 

We shall discuss in the next section that it is often possible to improve the rate of convergence towards $f_{\infty,V}$ by adding a non-reversible perturbation in \eqref{eq:FPsymmetric} while preserving the steady state $f_{\infty,V}$ (as done in \cite{GM} and \cite{LNP}). 
As a first step we consider 
the non-symmetric FP-equation
\begin{equation}
\label{eq:PDEwithB}
\partial_t f_t=\operatorname{div}_x{ (  \nabla_x f_t   + (\nabla_x V +b)f_t )}=:-L_b f_t,
\end{equation}
with $b=b(x)$ such that $\operatorname{div}_x{(b e^{-V})}=0$, to keep the steady state condition \linebreak
$L_{b}f_{\infty,V}=0$ still valid. 
In this paper we will only consider Fokker-Planck equations with linear drift, just as in \cite{LNP, GM}. This   corresponds to 
quadratic potentials
\begin{equation}
\label{eq:linearcase}
V(x)=\frac{x^T K^{-1} x}{2}, \ K \in \SM 
\end{equation} 
and linear perturbations of the form 
$b=Ax, \ A \in \mathcal{M}$.\\

{\bf Notation:}
Here and in the sequel we denote with $\mathcal{M}$ the set of real $d\times d$ matrices, $\SM$ (resp. $\SMU$)
the set of positive definite (resp. positive semi-definite) symmetric matrices, and with $\mathcal{A}$ the set of anti-symmetric matrices. The spectrum of $A \in \mathcal{M}$ is denoted by $\sigma(A)$. 
\bea{For a symmetric matrix $A$, $\lambda_{min}(A)$ and $\lambda_{max}(A)$ denote, respectively, its smallest and largest eigenvalue.}\\

The following lemma (Lemma 1 in \cite{LNP}) characterizes explicitly the admissible perturbations in \eqref{eq:PDEwithB}.
\begin{lemma}
Let $V(x)$ be given by \eqref{eq:linearcase} and let $b(x)=Ax$ where $A \in \mathcal{M}$. Then
\begin{equation}
\operatorname{div}_x{(be^{-V})}=0 \text{ if and only if } A=JK^{-1} \text{ with some }\ J \in \mathcal{A}.
\end{equation}
\end{lemma}
Then the non-symmetric Fokker-Planck equation \eqref{eq:PDEwithB} becomes 
\begin{equation}
\label{eq:FPwith J}
\partial_t f_t = \operatorname{div}_x{( \nabla_x f_t  +(I_d+J)K^{-1} xf_t )}, 
\end{equation}
where $I_d$ denotes the identity matrix in $\mathcal{M}$, $J \in \mathcal{A}$ is arbitrary, and $f_{\infty, V}$ is still a steady state.

Note that \eqref{eq:FPwith J} still satisfies \eqref{eq:coercRevFP} with the same rate $\lambda$ (see \S2.4 in \cite{AMTU}), but $\lambda$ may be smaller than the spectral gap of $L_b$. However, the sharp decay rate can be recovered by hypocoercivity tools \cite{Vi09, AE}: Then one finds constants $\tilde{\lambda}>0$ and $c\geq 1$ (depending on the fixed potential $V$ and the matrix $J$) such that
\begin{equation}
\label{eq:hypowithJ}
||f_t-f_{\infty, V}||_{\mathcal{H}} \leq c e^{-\tilde{\lambda}t} ||f_0-f_{\infty, V}||_{\mathcal{H}}, \quad \forall t \geq 0.
\end{equation}
For the reversible FP-dynamics (with $b=0$) the maximal decay rate $\tilde{\lambda}$ in the estimate \eqref{eq:hypowithJ} is $\lambda$, the biggest coercivity constant in the inequality \eqref{eq:coervlam}. In this case, the multiplicative constant $c=1$. The advantage of adding a non-reversible perturbation $b$ is to possibly obtain a larger decay rate $\tilde{\lambda}>\lambda$, at the price of allowing for a multiplicative constant $c>1$. In fact, the decay rate may be improved iff $K$ is not a multiple of $I_d$, see \S3.2, \cite{LNP}.

The question discussed in \cite{LNP} is the following: Given the potential \eqref{eq:linearcase}, which is the optimal non-reversible linear FP-equation of the form \eqref{eq:FPwith J} \bea{(and with {\it time-independent} coefficients)} such that its solutions converge to $f_{\infty, V}$ with largest decay rate? For the diffusion matrix fixed as $I_d$, as in \eqref{eq:FPwith J}, the authors give a complete answer in \cite{LNP}, Theorem 1. But if one generalizes the question, allowing to vary both $b$ and the diffusion matrix, as done in \cite{GM}, the best decay rate from \cite{LNP} can be improved further. \bea{Finally, one can extend the question further and analyze if {\it time-dependent} coefficients can enhance the decay of linear FP-equations even more.}

\bea{Let us put this paper more into context with the literature on entropy methods and hypocoercivity: The main goal of \cite{AAC, AAS, AMTU, ASS, AE, Vi09} is to find explicit and sometimes even optimal decay rates for {\it a given evolution equation}. 
By contrast, the novelty in \cite{LNP, GM} and here is to {\it fix an equilibrium density} and then to {\it seek the evolution equation} (within a certain class) that yields the fastest convergence towards the equilibrium.}

This paper is organized as follows: 
In the next section we formulate this optimization problem and review the results from \cite{GM}. In \S\ref{sec3} we present the main result: As the biggest improvement compared to \cite{GM}, we shall be able to obtain multiplicative constants in \eqref{eq:hypowithJ} that are arbitrarily close to 1. 
\bea{In \S\ref{sec:4.1}, \S\ref{sec:4.2} we will elucidate this result on 2D examples, giving sharp decay estimates and numerical illustrations. Moreover, we identify the non-symmetric perturbation of the FP-equation as a highly rotating drift term. Then, in \S\ref{sec:4.3} we discuss the issue of using time-dependent coefficient matrices to accelerate the decay behavior, mostly focussing on a numerical case study in 2D. Finally, we conclude in \S\ref{sec:5}.
}

\section{Formulation of the optimization problem and existing results}\label{sec2}
Let $K\in \SM$ be given.
We define the (typically) anisotropic Gaussian 
\begin{equation}
\label{eq:finf}
f_{\infty, K}(x):= \frac{\det(K)^{-1/2}}{(2 \pi)^{d/2} }\mathrm{exp}{\left(-\frac{x^TK^{-1}x}{2}\right)}, \qquad x \in \R^d,
\end{equation}
and the linear Fokker-Planck equation
\begin{equation}
\label{eq:FP}
 \partial_{t} f_t = -L_{C,D}f_t := \mathrm{div}_{x}{(D\nabla_x f_t+Cxf_t)}, \qquad x \in \R^d, \ t \in (0,\infty),
\end{equation}
for arbitrary $x$-independent matrices $D \in \SMU$ and $C \in \mathcal{M}$.
Equation \eqref{eq:FP} is a generalization of the non-reversible \eqref{eq:FPwith J}, possibly with a degenerate (i.e. singular) diffusion matrix $D$. 
Moreover we define the set
\begin{equation}
\mathcal{I}(K):= \{ (C,D) \in \mathcal{M} \times \SMU, \  \mathrm{Tr}{(D)} \leq d : L_{C,D}f_{\infty, K}=0 \}.
\end{equation}
The next lemma (Lemma 3.1 in \cite{GM}; for $D=I_d$ also Lemma 1 in \cite{LNP}) gives a characterization of the pairs $(C,D)$ in $\mathcal{I}(K)$.
\begin{lemma}
\label{lem:IKchar}
For $K \in \SM$ fixed, the following two statements are equivalent:
\begin{itemize}
\item $(C,D) \in \mathcal{I}(K)$;
\item $D \in \SMU, \mathrm{Tr}{(D)} \leq d,$ and $\exists J \in \mathcal{A}$ such that $C=(D+J)K^{-1}.$
\end{itemize}
\end{lemma}

In other words, for $K \in \SM$ given, we have
\begin{equation}
\label{eq:IK2}
\mathcal{I}(K)= \{ ((D+J)K^{-1}, D) \in \mathcal{M} \times \SMU : \ J \in \mathcal{A}, \mathrm{Tr}{(D)} \leq d \},
\end{equation}
and $\mathcal{I}(K)$ is not empty. 

Given a fixed covariance matrix $K \in \SM$ (and hence the fixed Gaussian $f_{{\infty},K})$, the set $\mathcal{I}(K)$ represents the matrix pairs $(C,D)$ such that their associated FP-equation admits $f_{\infty, K}$ as a normalized steady state. But reversely, for such a FP-equation, the (normalized) steady state $f_{\infty, K}$ does not have to be unique (e.g.\ $C=D=\mathrm{diag}{(1,0)}$ admits \eqref{eq:finf} with any $K=\mathrm{diag}{(1,\kappa)}$, $\kappa>0$).
It is known from the literature (see for example Theorem 3.1, \cite{AE}) that the existence of a unique $L^1-$normalized steady state $f_{\infty, K}$ for \eqref{eq:FP} is equivalent to the following two conditions on $(C,D) \in \mathcal{M} \times \SMU$:
\begin{enumerate}
\item $C$ is positive stable (i.e., \amit{$C$ has a positive spectral gap $\rho(C):=\min\{\Re (\lambda)\::\: \lambda\in\sigma(C)\}$});
\item hypoellipticity of \eqref{eq:FP} (i.e., there is no non-trivial $C^T$-invariant subspace of $\mathrm{ ker}{(D)}$).
\end{enumerate}

For our set-up, hypoellipticity can actually be inferred from $(C,D)\in\mathcal{I}(K)$; and more precisely:
\begin{lemma}\label{lem:2.2}
For some fixed $K \in \SM$, let $(C,D)\in \mathcal{I}(K)$ and $\rho(C)>0$. Then the corresponding FP-equation \eqref{eq:FP} is hypoelliptic.
\end{lemma}
\begin{proof}
Normalized steady states of \eqref{eq:FP} are Gaussian with its covariance matrix $Q$ satisfying the continuous Lyapunov equation
\begin{equation}\label{contLy}
  2D = CQ+QC^T\ .
\end{equation}
Since $D\ge0$ and $\rho(C)>0$, \eqref{contLy} has a unique, symmetric and positive semi-definite solution $Q$ (see, e.g., Theorem 2.2 in \cite{SZ}), namely $Q=K$.

By the above mentioned equivalence to the uniqueness of the steady state, \eqref{eq:FP} is hypoelliptic.
\end{proof}

For each fixed steady state $f_{\infty, K}$ we now want to answer the following questions: 
\begin{enumerate}[(Q1)]
\item \label{Q1} Which FP-evolution(s) converge(s) the fastest, i.e. with largest decay rate $\lambda_{opt}$ to the steady state in the operator norm of $e^{-L_{C,D}t}$ on $V_0^\perp\subset\mathcal{H}:=L^2(\R^d, f_{\infty, K}^{-1})$? 
\item \label{Q2} Second, when the best decay rate is fixed, what is the infimum of the multiplicative constant, $c_{inf}$, in the decay estimate \eqref{eq:hypowithJ}? 
\item \label{Q3} Third, for a fixed $K\in\SM$ and the corresponding $\lambda_{opt}$, and for any $c>c_{inf}$, which pair(s) of matrices $(C_{opt},D_{opt}) \in \mathcal{M} \times \in \SMU$ are such that $e^{-L_{C_{opt},D_{opt}}t}$ yields the convergence estimate \eqref{eq:hypowithJ} with the constants $(\lambda_{opt},c)$?
\item \label{Q4} Forth, for such an optimal pair of matrices, what bound on $C_{opt}$ can be found, and how does this bound grow w.r.t.\ to the space dimension $d$?
\item \label{Q5} \bea{Could something be gained by allowing for time-dependent matrices $C(t)$, $D(t)$?
}
\end{enumerate}

\begin{remark}We note that, without the additional constraint $\mathrm{Tr}{(D)} \leq d$ in the definition of $\mathcal{I}(K)$, the problem of finding an optimal evolution in the above sense would be ill-posed: Indeed, if $f_t$ converges to $f_{\infty,K}$ \amit{as $t\to\infty$, then $f^{\alpha}_t:=f_{\alpha t}$, for any $\alpha>0$ and} pertaining to $(\alpha C, \alpha D)$, has the same equilibrium and converges $\alpha$ times faster to it. For this reason, we shall only consider diffusion matrices with a prescribed bound for the trace, as in \cite{GM}. 
In probabilistic language it corresponds to the requirement that the upper bound on the total amount of randomness simultaneously injected in the system is prescribed, and this bound is equal to the case $D=I_d$.
\end{remark}

\bea{Next} we shall 
optimize the decay rate within the family of FP-equations \eqref{eq:FP} satisfying $(C,D) \in \mathcal{I}(K)$. 
But our choice of matrix $C$ will, in general, differ from the one constructed in \cite{GM}. We base this optimization on the fact that the sharp exponential decay rate of the FP-equation \eqref{eq:FP} equals $\rho(C)$ (at least for $C$ diagonalizable, see \cite{AE}, e.g.). Actually, \eqref{eq:FP} and its associated drift ODE, i.e. $\frac{d}{dt}x=-\widetilde{C} x$, with $\widetilde{C}:=K^{-1/2}CK^{1/2}$ \amit{(and hence $\rho(C)=\rho(\widetilde{C})$)}, have an even closer connection, as proven in Theorem 2.3, \cite{ASS}: 
\begin{theorem}
\label{theo:theoAAS}
Let $K\in\SM$ be given. 
We consider a FP-equation \eqref{eq:FP} with $(C,D) \in \mathcal{I}(K)$ and $C$ positive stable. 
Then, the propagator norms of \eqref{eq:FP} and of its corresponding drift ODE $\frac{d}{dt}x=-\widetilde{C}x$ are equal, i.e.
\begin{equation}
\label{eq:equalityASS}
  \left\|e^{-L_{C,D}t}\right\|_{\mathcal{B}(V_0^{\perp})} = \left\|e^{-\widetilde{C}t}\right\|_{\mathcal{B}(\R^d)}, \qquad \forall t \geq 0,
\end{equation}
where $||\cdot||_{\mathcal{B}(V_0^{\perp})}$ denotes the operator norm \bea{on $\mathcal{H}$} and orthogonality is considered w.r.t. $\mathcal{H}$. Moreover, 
$$
  ||A||_{\mathcal{B}(\R^d)}  := \sup_{0\neq x_0 \in \R^d}
  \frac{\left |\left | A x_0\right | \right |_2}{\left | \left | x_0 \right | \right |_2}
$$
denotes the spectral matrix norm of any matrix $A\in\mathcal M$. 
\end{theorem}
\bea{This result motivates to investigate the maximum spectral gap of $C$. Indeed, the 
next theorem (see Theorem 2.1 in \cite{GM}) identifies the maximum spectral gap of matrices of the form $C=(D+J)K^{-1}$, and its proof (in \cite{GM}) provides an explicit, algorithmic construction of a corresponding, optimal matrix pair $(C,D)$. 
\begin{theorem}
\label{theo:maxspectrum}
For $K \in \SM$ given,
\begin{equation}
\label{eq:maxspectrum}
\mathrm{max}{\{ \rho(C): \ (C,D) \in \mathcal{I}(K)\}}= \mathrm{max}(\sigma(K^{-1}))=\mathrm{min}(\sigma(K))^{-1}.
\end{equation}
\end{theorem}}

Concerning the above questions, the article \cite{GM} gives the following (partial) answers: The authors give a complete and positive answer to question Q\ref{Q1}, obtaining the optimal decay rate $\lambda_{opt}=\max(\sigma(K^{-1}))$. Their optimal pair  \linebreak
$(C_{opt},D_{opt}) \in \mathcal{I}(K)$ is very degenerate, the rank of $D_{opt}$ being one
\bea{(and this will also be the case for our approach below).}
But concerning questions Q\ref{Q2} and Q\ref{Q3}, they obtain an estimate for the multiplicative constant that grows dramatically with the dimension (in fact of order $d^{40 d^2}$). This is obtained in \cite{GM} when considering a  FP-equation with {\it time-independent} coefficients, i.e.\ the equation form introduced in \eqref{eq:FP}. As a remedy, the authors then considered {\it time-dependent} coefficients, using a symmetric FP-equation with the matrices $(K^{-1},I_d)$ for small times and a non-symmetric FP-equation for large times. 
\bea{Discontinuous coefficients were used there for analytical reasons, to improve decay estimates. But since their estimates are not sharp, it is not clear if time-dependent coefficients are really able to enhance the decay property of the {\it exact} FP-propagator norm, i.e.\ the true function of time, without estimates. We shall return to this question in \S\ref{sec:4.3} to elucidate question Q\ref{Q5}. \\
\indent
While the main result of \cite{GM}} is presented for the logarithmic relative entropy, the same argument works also for the $L^2$-norm, as already noted on page 5, \cite{GM}:
\begin{theorem}[Theorem 2.2, \cite{GM}]
\label{theo:theo2.2inGM}
\amit{Let $K\in\SM$ be given. }
\
\begin{enumerate}
\item[(a)] For any $\bea{\tilde c}>1$ it is possible to construct a matrix pair $(C_{opt},D_{opt}) \in \mathcal{I}(K)$ such that, for all normalized $f_0\in\H$ and for all $t_0>0$,
\begin{equation}
\label{eq:hypoestimatesinGM}
||f_t-f_{\infty,K}||^2_{\mathcal{H}} \leq \bea{\tilde c} \frac{\max(\sigma(K^{-1}))}{2t_0 } e^{-2\max(\sigma(K^{-1}))(t-t_0)} ||f_0-f_{\infty}||^2_{\mathcal{H}}\,, \qquad t\geq t_0, 
\end{equation}
where $f_t$ solves the following system of FP-equations
\begin{equation}\label{splitFP}
    \begin{cases}
    \partial_t f_t=\operatorname{div}_x( \nabla_x f_t + K^{-1}xf_t), \qquad\qquad\:\:\:\: 0\leq t \leq t_0\,,
    \\
     \partial_t f_t=\operatorname{div}_x( D_{opt}\nabla_x f_t + C_{opt}xf_t), \qquad t > t_0\,.
    \end{cases}\,
\end{equation}
\item[(b)] For the choice $\bea{\tilde c}=2$ in part (a), the matrix 
$C_{opt}$ can be estimated as
\begin{equation}\label{C-Frobenius}
||C_{opt}||_\F \le 4d^2 \sqrt{\kappa(K)}\,\lambda_{opt} \,,
\end{equation}
where $|| \cdot ||_{\mathcal{F}}$ denotes the Frobenius norm $||A||_{\mathcal{F}}:=\sqrt{\mathrm{Tr}\left( A^TA \right)}$, and $\kappa(K)$ is the condition number of $K$.
\end{enumerate}
\end{theorem}

Optimizing the estimate \eqref{eq:hypoestimatesinGM} w.r.t.\ the switching time $t_0$, and using the trivial bound $\left\|e^{-L_{C,D}t}\right\|_{\mathcal{B}(V_0^{\perp})}\le1$ we obtain the following result:
\begin{corollary}\label{cor-GM}
Under the assumptions of Theorem \ref{theo:theo2.2inGM}, and when choosing $t_0:=\min(\sigma(K))/2$, the following estimate holds for all normalized $f_0\in\H$:
\begin{equation}
\label{eq:hypoestimatesinGM2}
||f_t-f_{\infty,K}||^2_{\mathcal{H}} \leq ||f_0-f_{\infty}||^2_{\mathcal{H}} \times
    \begin{cases}
    1, \qquad\qquad\qquad\qquad\qquad\qquad\qquad\quad\:\:\:  0\leq t \leq t_0\,,
    \\
     \min\big\{1,\,\bea{\tilde c}\,\kappa(K)\,e^{1-2\max(\sigma(K^{-1}))t}\big\}, \qquad t > t_0\,.
    \end{cases} 
\end{equation}
where $f_t$ solves the FP-system \eqref{splitFP}.
\end{corollary}

Hence, Theorem 2.2 from \cite{GM} can only yield multiplicative constants $c\bea{=\sqrt{\tilde c\,\kappa(K)\,e}}>\sqrt{\kappa(K)\,e}$, using the notation of \eqref{eq:hypowithJ}.
\bigskip

In the next section we shall improve this result in three directions: Answering question Q\ref{Q2} we shall prove that $c_{inf}$ is always 1, and concerning question Q\ref{Q3} we shall construct an optimal matrix pair $(C_{opt}(c),D_{opt}(c))$ for any given $c>1$. Moreover, we shall not need to split the FP-evolution in time, in contrast to \eqref{splitFP}. Our key ingredient to obtain an improved result (compared to \cite{LNP, GM}) is the equality of the propagator norms of the FP-equation and of its drift ODE, see Theorem \ref{theo:theoAAS}. This reduces the quest for an optimal decay estimate to an analogous, and hence easier ODE problem, without having to invoke a hypocoercive entropy method as in the proof of Theorem 2.2, \cite{GM}, or the block-diagonal decomposition of the FP-propagator as in the proof of Proposition 11, \cite{LNP}. Finally, concerning question Q\ref{Q4} we shall show that our drift matrix $C_{opt}(c)$ grows like $\mathcal O(d^{3/2})$ (for any fixed $c>1$), compared to an $\mathcal O(d^2)$--growth in \cite{GM}.

\bea{
\subsection{Time-dependent coefficients}\label{sec2.1}
In order to analyze also the decay behavior of the split FP-equation \eqref{splitFP}, we shall next admit in the FP-equation \eqref{eq:FP} time-dependent coefficient matrices:
\begin{equation}\label{eq:tFP}
 \partial_{t} f_t = -L(t)f_t := \mathrm{div}_{x}{(D(t)\nabla_x f_t+C(t)xf_t)}, \qquad x \in \R^d, \ t \in (0,\infty).
\end{equation}
Here we assume that each FP-operator $L(t)$, with $t\ge0$ fixed, admits $f_{\infty,K}$ as a steady state, and that the covariance matrix $K\in\SM$ is given and independent of $t$. Hence, the coefficient matrices satisfy $(C(t),D(t))\in \I(K)$ $\forall t\ge0$, and by Lemma \ref{lem:IKchar}:
$$
  C(t)=(D(t)+J(t))\,K^{-1}, \quad \mbox{with some } J(t)\in\AAA,\:\:
  \forall t\ge0.
$$
We shall assume $\forall t\ge0$ that $\rho(C(t))>0$. Hence, by Lemma \ref{lem:2.2} each FP-operator $L(t)$ is hypocoercive. For \eqref{eq:tFP}, Theorem \ref{theo:theoAAS} can be extended: In the following theorem $S(t_2,t_1)$ and $T(t_2,t_1)$, $0\le t_1\le t_2<\infty$ will denote, respectively, the propagator operators for the PDE \eqref{eq:tFP} and the ODE \eqref{eq:tODE} that map an initial condition at time $t_1$ to the solution at time $t_2$.\\
\begin{theorem}\label{theo:t-dep}
  Let $K\in\SM$ be given. Let $(C(t),D(t))\in \I(K)$  be piecewise smooth functions of $t\ge0$ (where points of discontinuity do not accumulate), such that the initial value problem for \eqref{eq:tFP} admits a unique solution in $C([0,\infty);\H)$. Then, the propagator norms of \eqref{eq:tFP} and of its corresponding drift ODE,
  \begin{equation}\label{eq:tODE}
    \frac{d}{dt}x=-\widetilde{C}(t)x,\quad t\in(0,\infty),
  \end{equation}
where $\widetilde{C}(t) := K^{-1/2}\,C(t)\,K^{1/2}$, are equal, i.e.:
  \begin{equation}\label{prop-equal-t}
     \left\|S(t_2,t_1)\right\|_{\mathcal{B}(V_0^{\perp})} = \left\|T(t_2,t_1)\right\|_{\mathcal{B}(\R^d)}, \qquad \forall 0\le t_1\le t_2<\infty.
  \end{equation}
\end{theorem}
Since this result is a straightforward extension of Theorem 2.3 in \cite{ASS}, we shall give only some hints on the notational differences in the Appendix.
}

\begin{section}{Main result}\label{sec3}
The next theorem is the main result of this work. It states the existence of pairs $(C_{opt},D_{opt})=(C_{opt}(c),D_{opt}(c)) \in \mathcal{I}(K)$ that yield the maximum decay rate of the propagator norm of $e^{-L_{C_{opt},D_{opt}}t}$, and in parallel yielding a multiplicative constant $c$ arbitrarily close to $1$. 
\begin{theorem}
\label{theo:maintheo} Let $K \in \SM$ be given.
\
\begin{enumerate}
\item[(a)] 
Then, for any constant $c>1$ there exists a pair $(C_{opt},D_{opt})$ $=$\linebreak 
$(C_{opt}(c),D_{opt}(c)) \in \mathcal{I}(K)$ such that
\begin{equation}
\label{eq:mainIn}
\left | \left | e^{-L_{C_{opt},D_{opt}}t} \right | \right | _{\mathcal{B}(V_0^{\perp})} \leq  c e^{-\max (\sigma(K^{-1}))t}, \qquad t \geq 0.
\end{equation}
\item[(b)] The matrices from part (a) can be estimated as
\begin{equation}
\label{eq:InFrobenius}
\left | \left |C_{opt} \right | \right | _{\mathcal{F}} \leq
\lambda_{opt} \, \Big[d +\sqrt{\kappa(K)} \, \frac{2\pi c^2}{\sqrt{3}(c^2-1)} \, \sqrt d\,(d-1) \Big],  \quad \left | \left | D_{opt} \right | \right | _{\mathcal{F}}=d.
\end{equation}
\end{enumerate}
\end{theorem}
In the proof we shall build upon the strategy from $\S 3$ in \cite{LNP}, and only deviate from their strategy in Step 2 below. Nevertheless we outline the full proof, to make it readable independently. 
\begin{proof}[Proof of Theorem \ref{theo:maintheo}(a).]
We recall that, given any matrix pair $(C,D)$ in $\IK$, we can rewrite the drift matrix $C$ (see Lemma \ref{lem:IKchar}) as 
$$
C=(D+J)K^{-1}=K^{1/2}(\widetilde{D} + \widetilde{J}) K^{-1/2},
$$
where $\widetilde{D}:= K^{-1/2} D K^{-1/2}$ and $\widetilde{J}=K^{-1/2}J K^{-1/2}$.
Moreover it is easy to check that the map $M \mapsto K^{-1/2} M K^{-1/2}$ is a bijection that leaves $\SMU$ and $ \mathcal{A}$ invariant.
We split the proof into three steps.

\textbf{Step 1}
We shall construct an optimal pair ($\widetilde{D}_{opt}, \widetilde{J}_{opt})$ and 
investigate the propagator norm of the ODE-evolution
\begin{equation}
\label{eq:ODEtilde}
\frac{d}{dt} x=- \widetilde{C}_{opt}x, \quad x_0:=x(0) \in \R^d, \ t \geq 0,
\end{equation}
where $\widetilde{C}_{opt}:= \widetilde{D}_{opt}+\widetilde{J}_{opt}$.
More precisely, we shall provide a decay estimate for
$|| e^{-\widetilde{C}_{opt}t } ||_{\mathcal{B}(\R^d)}$
by constructing an appropriate Lyapunov functional (following \S2.1 of \cite{AAC}).

\bea{Following the proof of Theorem 2.1 in \cite{GM} we recall that $D$ can enable the maximum decay rate $\lambda_{opt} :=  \max{ (\sigma(K^{-1})) }$, only if the range of $D$ is a subset of $\Omega$, i.e. the eigenspace of $K^{-1}$ corresponding to $\lambda_{opt}$. Hence we let}
 $v \in \R^d$ be a normalized eigenvector of $K^{-1}$ associated to $\lambda_{opt}$. As in \cite{GM} we define the rank-1 matrix $D_{opt}:=d (v \otimes v) \in \SMU$ with $\mathrm{Tr}(D_{opt})=d$.  It follows that
\begin{equation}
\label{eq:Dtilde2}
\widetilde{D}_{opt}= d K^{-1/2}(v \otimes v) K^{-1/2}=d \lambda_{opt} (v \otimes v) =\lambda_{opt} D_{opt}, 
\end{equation}
and hence 
$$\frac{ \mathrm{Tr}\left(\widetilde{D}_{opt}\right)}{d}=\lambda_{opt}.$$
\bea{For $\mathrm{dim}(\Omega) >1$, we remark that the choice of $\widetilde{D}_{opt}$ made in \eqref{eq:Dtilde2} is just \textit{one} simple option,  which enables the decay rate $\lambda_{opt}$.
}
For the construction of $\widetilde{J}_{opt} \in \mathcal{A}$ we use a particular basis of $\R^{d}$:
Let $\{ \psi_k \}_{k=1}^d$ be an orthonormal basis of $\R^d$ such that the following condition (Lemma 2, \cite{LNP}) is satisfied: for all $k\in \{1,...,d\}$,
\begin{equation}
\label{eq:cond39}
\langle \psi_k, \widetilde{D}_{opt} \psi_k \rangle= \frac{\mathrm{Tr}\left(\widetilde{D}_{opt}\right)}{d}=\lambda_{opt}.
\end{equation}
The existence of such basis is guaranteed by Proposition 3 in the same paper. 
The essence of the basis $\{\psi_k\}_{k=1}^d$ is to provide an equidistribution of $\mathrm{Tr}\left(\widetilde{D}_{opt}\right)$  into the directions $\{\psi_k\}_{k=1}^d$, while $\widetilde{D}_{opt}$ has only rank 1. This is the starting point to enable a uniform (in $x_0$ and $t$) decay estimate of all trajectories of \eqref{eq:ODEtilde},  see \eqref{eq:decayfortheODE} below.
We observe that in \cite{LNP} the hypotheses of Proposition 3 require $\widetilde{D}_{opt}$ to be invertible.  However,  this condition can be weakened to $\widetilde{D}_{opt} \in \SMU$, as already pointed out in \cite{GM}: $\widetilde{D}_{opt}+\epsilon I_d \in \SM$, and $\epsilon \rightarrow 0^+$  yields the above result. 

Next, let $0<\lambda_1< \cdots < \lambda_d$ be arbitrary numbers in $\R$ that will be chosen later in a suitable way.
We define the matrix $\widetilde{J}_{opt}:= \Psi \widehat{J}_{opt} \Psi^{-1} \in \mathcal{A}$, with $\Psi:=\left [\psi_1,...,\psi_d \right ]$ and $\widehat{J}_{opt}$ is the anti-symmetric matrix with elements (as in Lemma 2, \cite{LNP}):
\begin{equation}
\label{def:J}
\left (\widehat{J}_{opt} \right )_{j,k}:=\frac{\lambda_j+ \lambda_k}{\lambda_j-\lambda_k} \langle \psi_j, \widetilde{D}_{opt} \psi_k \rangle, \qquad \forall j \neq k, 
\end{equation}
and 0 else. 

Now, the strategy consists in finding a suitable symmetric matrix $P \in \SM$ that defines a modified norm $||\cdot||_P$ in $\R^d$ such that the trajectories of the ODE \eqref{eq:ODEtilde} decay with pure exponential decay rate $\lambda_{opt}$ w.r.t.\ this norm. 

\textbf{Step 2}
Let us proceed with the construction of its inverse matrix $Q:=P^{-1} \in \SM$. 
We define
$Q:= \Psi \Lambda \Psi^{-1}$, with the matrix $\Lambda:= \mathrm{diag}{(\lambda_1,...,\lambda_d)}$. We observe that $Q \in \SM$ due to the orthonormality of $\Psi$ and the positivity of $\lambda_i$. Moreover, by definition, $Q$ has the  eigenvectors $\psi_i$ and eigenvalues $\lambda_i$.
By using Lemma 2, \cite{LNP} (or a straightforward computation using \eqref{eq:cond39}) the following Lyapunov equation holds for $Q$:
\begin{equation}
\label{eq:lyap}
\widetilde{J}_{opt}Q-Q\widetilde{J}_{opt}=-Q \widetilde{D}_{opt}-\widetilde{D}_{opt}Q+2  \lambda_{opt} Q.
\end{equation}
Let us define the modified norm $||x||^2_P:=\langle x, Px \rangle$ on $\R^d$, where $P:=Q^{-1} \in \SM$.  Differentiating this norm along a trajectory of the ODE \eqref{eq:ODEtilde} we obtain with \eqref{eq:lyap}, multiplied on either side by $P=Q^{-1}$: 
\begin{equation}
\frac{d}{dt}||x(t)||^2_P=- \Big\langle x(t),\left[P \left(\widetilde{D}_{opt}+\widetilde{J}_{opt} \right)+\left(\widetilde{D}_{opt}-\widetilde{J}_{opt}\right)P\right] x(t) \Big\rangle 
=-2 \lambda_{opt} ||x(t)||^2_P.
\end{equation}
Hence the modified norm decays with rate $\lambda_{opt}$, i.e.
\begin{equation}
\label{eq:decayfortheODE}
||x(t)||_P^2 =e^{-2\lambda_{opt} t } ||x(0)||_P^2, \qquad t \geq 0.
\end{equation}
Transforming to the Euclidean vector norm, we obtain for the propagator 
\begin{equation}
\label{eq:decayEucl}
\left | \left |e^{-\widetilde{C}_{opt}t} \right | \right |_{\mathcal{B}(\R^d)} \leq \sqrt{\kappa (P)}\,e^{- \lambda_{opt} t}, \qquad t \geq 0,
\end{equation}
where $\kappa(P)$ denotes the condition number of the matrix $P$.

\textbf{Step 3}
The multiplicative constant appearing in \eqref{eq:decayEucl} can be adjusted by choosing the eigenvalues of $P$ in the following way: Given any $c>1$, and due to the fact that $\kappa(P)=\kappa(Q)=\frac{\lambda_d}{\lambda_1}$, it is sufficient to choose $\lambda_d$ and $\lambda_1$ such that their quotient is (less or) equal to $c^2$. The remaining parameters $\lambda_2<...<\lambda_{d-1}\in(\lambda_1,\lambda_d)$ could be freely chosen at this point, but assigning them a precise value will be crucial in the proof of part (b).

To summarize, we have proved so far that, for any prescribed $c>1$, there exists a pair  of matrices $\widetilde{J}_{opt} \in \mathcal{A}$ and $\widetilde{D}_{opt} \in \SMU$ such that
\begin{equation}
\label{eq:optestimateOde}
\left | \left |e^{-\widetilde{C}_{opt}t} \right | \right |_{\mathcal{B}(\R^d)}  \leq c e^{-\lambda_{opt} t}, \qquad t \geq 0.
\end{equation}
We conclude the proof by combining Theorem \ref{theo:theoAAS} applied to the operator
\\
$e^{-L_{C_{opt},D_{opt}}t}$, and the above inequality \eqref{eq:optestimateOde}.
\end{proof}
\bea{For $K= \alpha I_d$, we remark that a trivial modification of the above proof admits the choice $(C_{opt}, D_{opt})=(K^{-1}, I_d)$, $J=0$,  $P=I_d$. In this case the reversible dynamics is already optimal with $\lambda_{opt}=\alpha^{-1}$ and $c=1$ in \eqref{eq:mainIn}. Moreover, $\left | \left | C \right | \right |_{\mathcal{F}}=\lambda_{opt} \sqrt{d}$.
}

\begin{proof}[Proof of Theorem \ref{theo:maintheo}(b).]
First we compute the Frobenius norm of $D_{opt}:=d (v \otimes v)$, with $v \in \R^d$ normalized eigenvector of $K^{-1}$:
\begin{equation}
\label{eq:estimFrobD}
\left | \left | D_{opt} \right | \right | ^2_{\mathcal{F}}=d^2\,\mathrm{Tr} \left( D_{opt}^2 \right) 
=d^2||v||_2^4=d^2.
\end{equation}

For estimating $\left | \left |C_{opt} \right | \right | _{\mathcal{F}}$ we recall 
$C_{opt}=D_{opt} K^{-1} + K^{1/2} \widetilde{J}_{opt} K^{-1/2}$, which implies
using the inequality $||A\,B||_\F\le ||A||_\F ||B||_{\mathcal{B}(\R^d)}$ (see \cite{HJ}, p.\ 364):
 \begin{equation}
 \label{eq:linkFrobtilde}
 \left | \left |C_{opt} \right | \right | _{\mathcal{F}} \leq 
\left | \left |D_{opt} \right | \right | _{\mathcal{F}} \max(\sigma(K^{-1}))+ \sqrt{\kappa(K)}\left | \left |\widetilde{J}_{opt} \right | \right | _{\mathcal{F}} \,.
 \end{equation}
%
Since the Frobenius norm is unitarily invariant and $\widetilde{J}_{opt}=\Psi \widehat{J}_{opt} \Psi^{-1}$ we have  $\left | \left |\widetilde{J}_{opt} \right | \right | _{\mathcal{F}}=\left | \left |\widehat{J}_{opt} \right | \right | _{\mathcal{F}}$. 
For any $k =1,...,d$ we define $\alpha_k:=\langle v , \psi_k \rangle$ and we observe that $\alpha_k^2=\frac{1}{d}$:  Indeed from \eqref{eq:cond39} it follows that
$$
\lambda_{opt}=\langle \psi_k , \widetilde{D}_{opt} \psi_k \rangle = d \lambda_{opt} \alpha_k^2.
$$
Hence we can rewrite \eqref{def:J} as 
\begin{equation}
\left( \widehat{J}_{opt} \right )_{j,k}
=\frac{\lambda_j+\lambda_k}{\lambda_j-\lambda_k} d \lambda_{opt} \alpha_j \alpha_k\,, \qquad \forall j \neq k. 
\end{equation}
It follows that 
\begin{equation}
\label{eq:JhatFrob}
\left | \left |\widehat{J}_{opt} \right | \right | _{\mathcal{F}}^2
=d^2 \lambda_{opt}^2\sum_{j \neq k=1}^d \left(\frac{\lambda_j+\lambda_k}{\lambda_j-\lambda_k} \right )^2 \alpha_j^2 \alpha_k^2
=\lambda_{opt}^2\sum_{j \neq k=1}^d \left(\frac{\lambda_j+\lambda_k}{\lambda_j-\lambda_k} \right )^2 .
\end{equation}

Next we choose the parameters $\lambda_k$, $k=1,...,d$ as 
\begin{equation}\label{lambda-AS}
  \lambda_k:=\frac{d-1}{c^2-1} +k-1, 
\end{equation}
and they satisfy $0<\lambda_1<\cdots <\lambda_d$ and $\frac{\lambda_d}{\lambda_1} = c^2$. 
Moreover we have for $j\ne k$: $(\lambda_j+\lambda_k)^2  <(2\lambda_d)^2= 4 \left( \frac{c^2}{c^2-1} \right )^2 (d-1)^2$, which implies with \eqref{eq:JhatFrob}: \begin{equation}
\label{eq:hatJsum}
\left | \left | \widehat{J}_{opt} \right | \right | _{\mathcal{F}}^2 \leq \lambda_{opt}^2 4 \left( \frac{c^2}{c^2-1} \right )^2 (d-1)^2 \sum_{ j\neq k=1}^d \frac{1}{(\lambda_j-\lambda_k)^2}.
\end{equation}
With the following estimate of a hyperharmonic series
\begin{equation*}
\sum_{j \neq k=1}^d \frac{1}{(\lambda_j-\lambda_k)^2}=\sum_{j=1}^d \sum_{\substack{k=1 \\ k\neq j}}^d \frac{1}{(j-k)^2}=\sum_{j=1}^d \left ( \sum_{l=1}^{j-1} \frac{1}{l^2} + \sum_{l=1}^{d-j} \frac{1}{l^2} \right ) \leq \sum_{j=1}^d \frac{\pi^2}{3} =d \frac{\pi^2}{3},
\end{equation*}
and \eqref{eq:hatJsum} we obtain
\begin{equation}
\label{eq:lasthatJopt}
\left | \left | \widehat{J}_{opt} \right | \right | _{\mathcal{F}}\leq \lambda_{opt} \frac{2\pi}{\sqrt{3}}   \frac{c^2}{c^2-1} \sqrt d(d-1).
\end{equation}
We conclude the proof by combining the (in)equalities
\eqref{eq:estimFrobD}, \eqref{eq:linkFrobtilde}, 
and \eqref{eq:lasthatJopt}.
\end{proof}

\bea{
Let us briefly compare the strategy of proofs for Theorem \ref{theo:maintheo} here and for Theorem 2.2 in \cite{GM}: The main difference concerns how to connect the evolution of the drift ODE to the FP-equation (here via the equality of the propagator norms, and via a hypocoercive entropy method in \cite{GM}). Further, our choice of the parameters $\lambda_k$ is (slightly) improved compared to the choice 
\begin{equation}\label{lambda-GM}
  \lambda_k=d+k, 
\end{equation}
in \cite[Remark 7]{LNP} and \cite{GM}. Finally, the proof of Theorem \ref{theo:maintheo}(b) provides a refined estimate of $\left | \left |C_{opt} \right | \right | _{\mathcal{F}}$.
}

\begin{remark}\label{rem:not-unique}
We note that, for any $c>1$, an optimal matrix pair $(C_{opt}(c),\,D_{opt}(c))$ is {\it not} unique: Using in the proof of Theorem \ref{theo:maintheo} the matrices $\widetilde{J}_{opt}^T$, $\widetilde{C}_{opt}^T$ instead of, respectively, $\widetilde{J}_{opt}$, $\widetilde{C}_{opt}$ and the norm $||\cdot ||_Q$ instead of $||\cdot ||_P$ yields another non-symmetric FP-equation that satisfies the same estimates \eqref{eq:mainIn}, \eqref{eq:InFrobenius}.
\end{remark}

\end{section}

\bea{
\section{Examples and numerical illustrations} \label{sec4}
In this section we shall illustrate the results of \S\ref{sec3}. For an explicit example in $\R^2$ we shall give a plot of the exact propagator norm for the FP-equation, which is accessible due to Theorem \ref{theo:theoAAS} for constant-in-time coefficients and due to Theorem  \ref{theo:t-dep} for the time-dependent case. First of all we shall illustrate Theorem \ref{theo:maintheo}(a), particularly focussing on the multiplicative constant in the exponential decay estimate \eqref{eq:mainIn}.
}

\bea{
\subsection{Optimal decay estimates}\label{sec:4.1}
As a first example we }
consider the covariance matrix $K=\mathrm{diag}(1,2) \in \R^{2 \times 2}$. Then the maximum decay rate for FP-equations that converge to $f_{\infty,K}$ is $\lambda_{opt}=\min(\sigma(K))^{-1}=1$. Next we shall construct one optimal pair of matrices $(C_{opt}, D_{opt})$ such that $e^{-L_{C_{opt},D_{opt}}t}f_0$ converges to $f_{\infty,  K}$ with decay rate $\lambda_{opt}$ and with a multiplicative constant arbitrary close to one.  For any $c>1$ we choose real numbers $0<\lambda_1< \lambda_2 $ such that $\frac{\lambda_2}{\lambda_1}= c^2$. We abbreviate 
\bea{$\mu:=\frac{\lambda_2+\lambda_1}{\lambda_2-\lambda_1}=\frac{c^2+1}{c^2-1}>1$.}  Following the procedure described in the proof of Theorem \ref{theo:maintheo}(a) we first compute $D_{opt}=\widetilde{D}_{opt}=\mathrm{diag}(2,0) \in \R^{2\times 2}$. An orthonormal basis of $\R^2$  satisfying condition \eqref{eq:cond39} is given by $\psi_1:=\frac{1}{\sqrt{2}}(1,1)^T$ and $\psi_2:=\frac{1}{\sqrt{2}}(-1,1)^T$. This defines the anti-symmetric matrix
\bea{
$\widetilde{J}_{opt}=\widehat{J}_{opt}=\begin{pmatrix}
0 & \mu
\\
-\mu & 0 
\end{pmatrix}.
$
Finally we compute 
\begin{equation}
C_{opt}=\begin{pmatrix}
2 & \frac{\mu}{\sqrt{2}}
\\
-\sqrt{2} \mu & 0 
\end{pmatrix},
\text{ \quad and \quad}
\widetilde{C}_{opt}=\begin{pmatrix}
2 & \mu
\\
 -\mu & 0
\end{pmatrix}.
\end{equation}
}

The spectral gaps of the drift matrix $C_{opt}$ and the operator $L_{C_{opt}, D_{opt}}$ coincide and are equal to 1. $C_{opt}$ has the two distinct eigenvalues $1\pm i\sqrt{\mu^2-1}$ (because \bea{$\mu>1$}), which are also eigenvalues of $L_{C_{opt}, D_{opt}}$ (see Theorem 5.3 in \cite{AE} or Proposition 10 in \cite{LNP}). Hence, $\lambda_{opt}=1$ is indeed the largest possible, uniform decay rate of the FP-propagator $e^{-L_{C_{opt}, D_{opt}}t}$ on $V_0^\perp$.

Thanks to Theorem \ref{theo:theoAAS} we can reduce the evaluation of the multiplicative constant $c$ in the decay estimate \eqref{eq:hypowithJ}
to the study of the propagator norm of the associated drift ODE $\frac{d}{dt}{x}=-\widetilde{C}_{opt}x$.  In Theorem 3.7,  \cite{AAS} the authors provide the explicit form of the best multiplicative constant for an ODE $\frac{d}{dt}{x}=-Ax$ in $\R^2$ when the matrix $A \in \R^{2\times 2}$ is positive stable, diagonalizable and $\Re \tau_1=\Re \tau_2 $, with $\tau_j$, $j=1,2$ the eigenvalues of $A$: Then the best constant $c_{min}$ in the exponential decay estimate for $e^{-At}$ is given by 
\begin{equation}
\label{eq:bestCinAAS}
c_{min}=\sqrt{\frac{1+\alpha}{1-\alpha}}, \qquad \alpha:=\left | \left  \langle \frac{v_1}{||v_1||_2} , \frac{v_2}{||v_2||_2} \right \rangle \right |,
\end{equation}
where $v_i \in \CC^2$, $i=1,2$ denote the eigenvectors of $A$.
Since $\widetilde{C}_{opt}$ satisfies the hypotheses of Theorem 3.7 in \cite{AAS},  
a straightforward computation gives $\alpha=\frac{1}{|\mu|}$, and the best multiplicative constant is $c_{min}=\sqrt{\frac{\lambda_2}{\lambda_1}}= c$, coinciding with the statement of Theorem \ref{theo:maintheo}(a). 

We observe that $c\searrow1$ implies \bea{$\mu\to\infty$}. This corresponds to the {\it high-rotational limit} in the drift matrix of the FP-equation 
$$
 \partial_{t} f_t = \mathrm{div}_{x}{(D_{opt}\nabla_x f_t+C_{opt}(\mu)xf_t)}.
$$
For increasing $|\mu|$, the latter tends to mix \bea{with increasing speed} the dissipative and non-dissipative directions (i.e.\ $x_1$ and $x_2$, respectively) of the corresponding symmetric FP-equation (i.e.\ with $\mu=0$).

As stated in Remark \ref{rem:not-unique}, replacing $\mu$ by $-\mu$ yields another FP-equation with the same optimal decay behavior. Only the rotational direction is then reversed.

\bea{
\subsection{Numerical illustrations: time-independent FP-equations}\label{sec:4.2}
To illustrate the construction of optimal coefficient matrices in Theorem \ref{theo:maintheo}(a) we revisit the 2D-example from \cite{GM}, i.e.\ $K=\diag(1/\eps,1)$, $\eps=0.05$ which admits the optimal decay rate $\lambda_{opt}=1$. For any given multiplicative constant $c>1$, the optimal coefficient matrices constructed in Theorem \ref{theo:maintheo}(a) read:
\begin{equation}\label{opt-matrices}
  D_{opt}=\widetilde D_{opt}=\begin{pmatrix}
0 & 0
\\
 0 & 2
\end{pmatrix}, \quad
  C_{opt}=\begin{pmatrix}
0 & -\frac{\mu}{\sqrt{\eps}}
\\
\sqrt{\eps} \mu & 2 
\end{pmatrix}, \quad 
\widetilde{C}_{opt}=\begin{pmatrix}
0 & -\mu
\\
 \mu & 2
\end{pmatrix},\quad \mu:=\frac{c^2+1}{c^2-1}.
\end{equation}
}

\begin{figure}[htbp]
\includegraphics[scale=0.7]{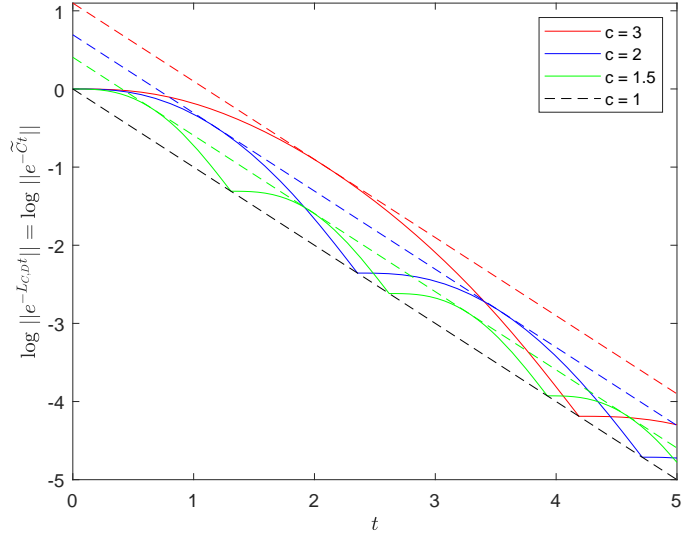}
\caption{\bea{The solid curves show the FP- and ODE-propagator norms as functions of $t$ for 3 values of the multiplicative parameter:  $c=3,\,2,\,1.5$ (top to bottom). The dashed curves give the corresponding (sharp) exponential bound of the form $c e^{-\lambda_{opt}t}$ for the 3 cases. The dashed black curve shows the exponential bound in the high-rotational limit, i.e.\ for $c\searrow1$. 
Colors only online.} }
\label{fig1}
\end{figure}

\bea{
In Figure\ref{fig1} we present the exact propagator norms (as a function of time) of the FP-equation and of its drift ODE, i.e.\
\begin{equation}\label{prop-norms}
    \left\|e^{-L_{C,D}t}\right\|_{\mathcal{B}(V_0^{\perp})} = \left\|e^{-\widetilde{C}t}\right\|_{\mathcal{B}(\R^d)}, \quad t\ge0
\end{equation}
for several prescribed values of the multiplicative constant: $c=1.5,\,2,\,3$. This figure includes also the r.h.s.\ of the corresponding exponential decay estimate \eqref{eq:mainIn}, using a logarithmic scale for the ordinate axis. Being the exact upper envelops, this reveals that this estimate is indeed sharp, concerning both the exponential rate and the multiplicative constant. 
Also note that each curve of the propagator norm periodically touches (from above) the curve corresponding to the high-rotational limit, given by $e^{-\lambda_{opt}t}$.\\
}

\bea{
Continuing with the same example, we shall next compare the results from Theorem \ref{theo:maintheo}(a) here and Theorem 2.2 in \cite{GM}. 
First we need to explain the criterion of comparison: 
For both results, and for a given constant $c>1$ we seek a matrix pair $(C,D)$ such that the inequality \eqref{eq:mainIn} holds. Since $c_{inf}=1$, such a pair can always be found, but $\|C\|_\F$ becomes large as $c\searrow1$ (see \S\ref{sec:4.1}). So, asking \eqref{prop-norms} to be close to the high-rotational limit $e^{-\lambda_{opt}t}$ {\it cannot} be a useful criterion. Instead, for given $c>1$ we want to find $(C,D)\in\I(K)$ such that \eqref{eq:mainIn} holds and $\|C\|_\F$ is minimal. This has also a practical implication for solving the FP-equation \eqref{eq:FP} numerically: $\|C\|_\F$ ``small'' allows to use ``large'' time steps.\\
\indent
For fixed $c=\sqrt 2$, Theorem \ref{theo:maintheo}(a) here and Theorem 2.2 in \cite{GM} yield, respectively, 
$$
  \widetilde{C}_{opt}^{AS}=\begin{pmatrix}
0 & -3
\\
 3 & 2
\end{pmatrix},\quad
\widetilde{C}_{opt}^{GM}=\begin{pmatrix}
0 & -7
\\
 7 & 2
\end{pmatrix},
$$
with $\|{C}_{opt}^{AS}\|_\F=\sqrt{184.45}$ and  
$\|{C}_{opt}^{GM}\|_\F=\sqrt{986.45}$~.
The essential difference stems from the different choices of $\lambda_1$ and $\lambda_2$, \eqref{lambda-AS} vs.\ \eqref{lambda-GM}. In Theorem \ref{theo:maintheo}(a), the estimate \eqref{eq:mainIn} is sharp, and hence the corresponding plot of the propagator norm has $\sqrt 2\,e^{-t}$, i.e.\ the r.h.s.\ in \eqref{eq:mainIn}, as its upper envelop (see Figure \ref{fig2}, left). Since the estimate from Theorem 2.2 in \cite{GM} is not sharp, the anti-symmetric part of $\widetilde{C}_{opt}^{GM}$ is larger than ``necessary'', and hence the corresponding plot of the propagator norm stays well below the estimate $\sqrt 2\,e^{-t}$. With a view towards numerical applications the latter is rather disadvantageous.}

\begin{figure}[htbp]
\begin{center}
\includegraphics[scale=0.42]{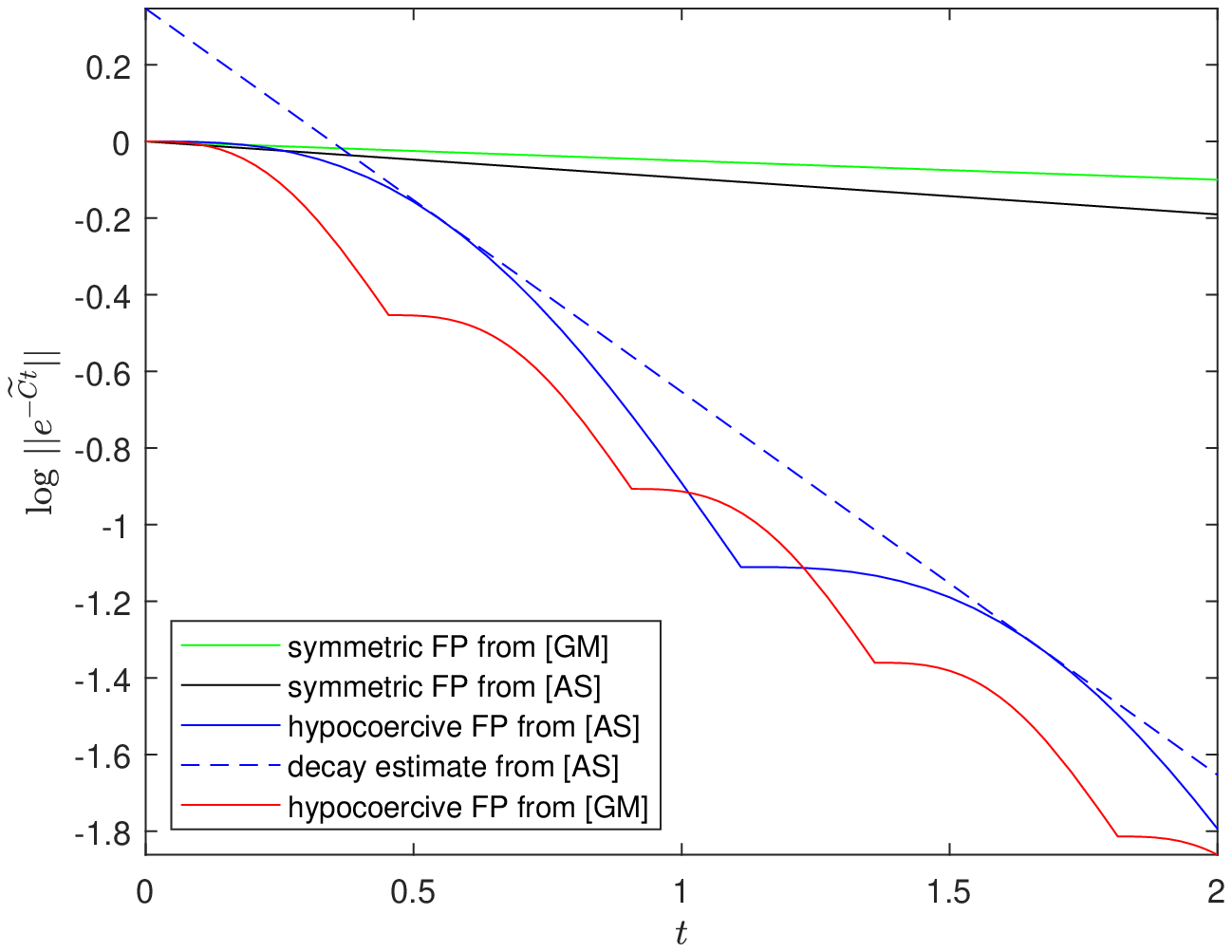}
\hfill
\includegraphics[scale=0.42]{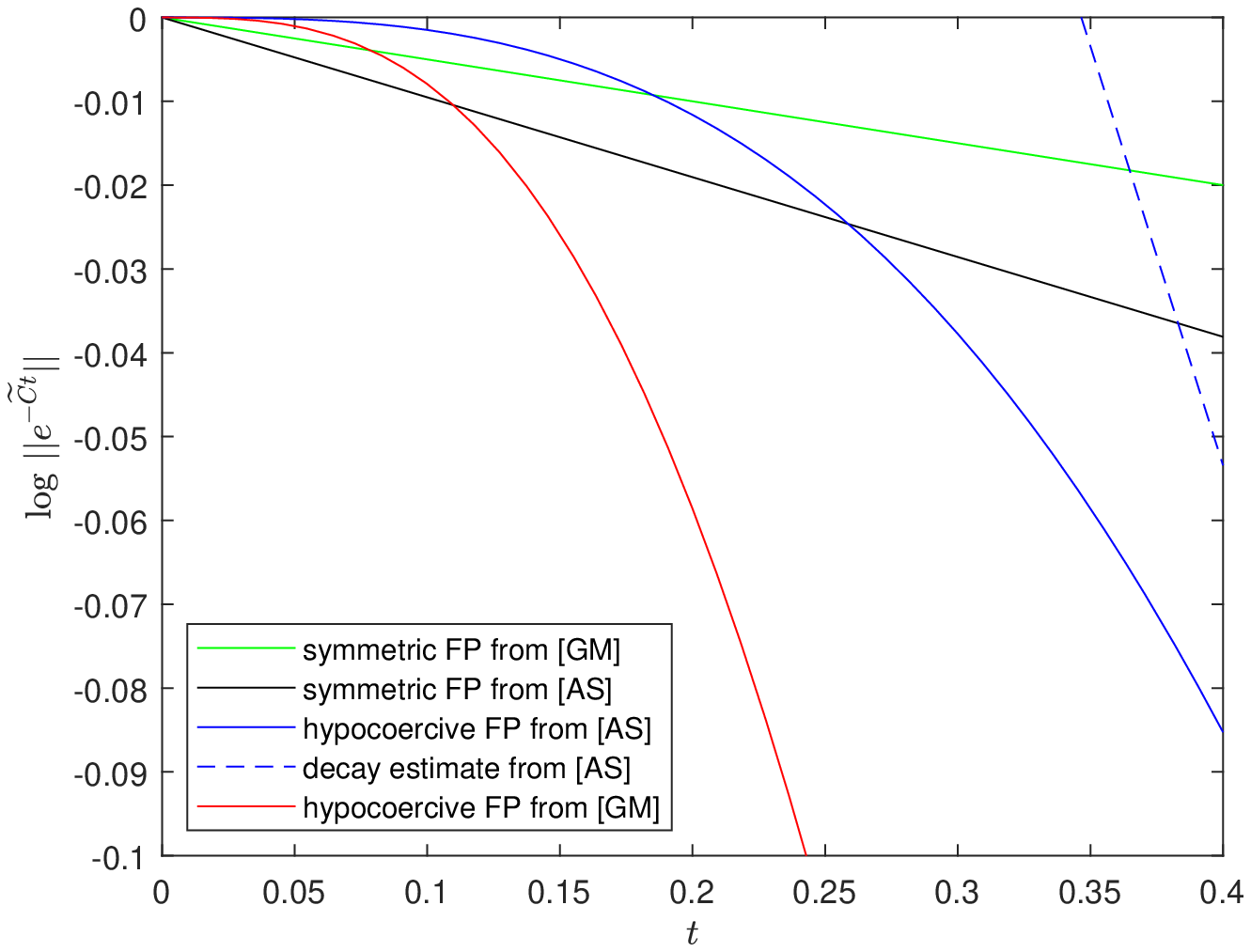}
\end{center}
\caption{\bea{Left: For $c=\sqrt2$, the solid blue and red curves show the FP- and ODE-propagator norms as functions of $t$ for the hypocoercive FP-equations constructed, respectively, in Theorem \ref{theo:maintheo}(a) here and Theorem 2.2 in \cite{GM}. The dashed blue curve gives the corresponding exponential bound $\sqrt2\, e^{-t}$; it is sharp for Theorem \ref{theo:maintheo}(a). The solid green and black curves show the FP- and ODE-propagator norms for the symmetric FP-equations in \eqref{splitFP} and \eqref{eq:FPopt}, respectively.\\
Right: a zoom of the plot, close to $t=0$. Colors only online.}  }
\label{fig2}
\end{figure}

\bea{
Figure \ref{fig2} also shows the decay of the propagator norm of the symmetric FP-equations in \eqref{splitFP} and \eqref{eq:FPopt}. Their respective decay rates are $\frac{1}{\lambda_{max}(K)}=\eps$ and $\frac{d}{\tr(K)}=\frac{2\eps}{1+\eps}$, both well below $\lambda_{opt}=1$, the rate of the optimal hypocoercive FP-equations.
}

\bea{
\subsection{Numerical illustrations: time-dependent FP-equations}\label{sec:4.3}
In \cite{GM} the authors used a FP-equation of the split form \eqref{splitFP} with piecewise constant coefficient matrices in order to approach the given equilibrium quickly. Following this approach, we shall next discuss if time-dependent coefficient matrices $C(t)$, $D(t)$ can accelerate the convergence in FP-equations, compared to the case of constant matrices $C$, $D$ that was analyzed in \S\ref{sec3}.\\
\indent
As a first step we shall show that the initial decay of hypocoercive FP-evo\-lu\-tions, as constructed in the proof of Theorem \ref{theo:maintheo}(a) (recall that $\rank(\widetilde D_{opt})=1$), can always be improved, e.g.\ in the spirit of the split FP-equation \eqref{splitFP} proposed in \cite{GM}. The following lemma gives, at $t=0$, the largest possible decay rate of the FP-equation \eqref{eq:FP} as well as of its drift ODE $\dot x=-\widetilde C x$ (both when considering their propagator norms).
\begin{lemma}\label{lem:0-decay}
Let $K\in\SM$ be given. For any $(C,D)\in\I(K)$, the maximum decay rate of $\left\|e^{-L_{C,D}t}\right\|_{\mathcal{B}(V_0^{\perp})}$ at $t=0$ equals $\frac{d}{\tr(K)}$. It is obtained by the symmetric FP-equation with $C=\frac{d}{\tr(K)}\,I_d$ and $D=\frac{d}{\tr(K)}\,K$.
\end{lemma}
\begin{proof}
Due to Theorem \ref{theo:theoAAS} we want to maximize the decay of the corresponding ODE-propagator norm,
$$
  \left\|e^{-\widetilde{C}t}\right\|_{\mathcal{B}(\R^d)}^2 = \lambda_{max}\left( e^{-\widetilde{C}^Tt} e^{-\widetilde{C}t} \right)
$$
at $t=0$. A Taylor expansion yields
\begin{equation}\label{taylor}
  \left\|e^{-\widetilde{C}t}\right\|_{\mathcal{B}(\R^d)} = 1 -\lambda_{min}(\widetilde C_s)\,t+\mathcal O(t^2)\quad \mbox{ as }t\to0,
\end{equation}
where $\widetilde C_s:=\frac12 (\widetilde C+\widetilde C^T)$ is the symmetric part of $\widetilde C$. We recall from the proof of  Theorem \ref{theo:maintheo}(a) that $\widetilde C:=K^{-1/2}CK^{1/2}$ and $\widetilde D:=K^{-1/2}DK^{-1/2}=\widetilde C_s\ge0$.\\
\indent
Thus we are led to the following optimization problem:
Find $\widetilde C_s\in\SMU$ with
\begin{equation}\label{trD}
  \tr(D)=\tr(K^{1/2}\widetilde C_s K^{1/2})=:\tau\le d,
\end{equation}
such that $\lambda_{min}(\widetilde C_s)$ is maximal. Since $\widetilde J$, the anti-symmetric part of $\widetilde C$, does not appear within this problem, we set it to 0, for simplicity.\\
\indent
For such an optimal $\widetilde C_s$, \eqref{trD} actually has to be an equality: Otherwise we would have
$$
  K^{1/2}\widetilde C_s K^{1/2}\le\tau\,I_d<d\,I_d
$$
and the matrix $\widetilde C_s$ could be ``enlarged'', e.g.\ by the matrix
$$
  A:=\frac{d-\tau}{d^2-\tau} (d\,K^{-1}-\widetilde C_s)\in\SM.
$$
Then, $\widetilde C_s +A$ still satisfies the constraint \eqref{trD}:
$$
  \tr\left(K^{1/2}[\widetilde C_s +A]K^{1/2}\right)=d,
$$
but $\lambda_{min}(\widetilde C_s+A)>\lambda_{min}(\widetilde C_s)$, contradicting the optimality of $\widetilde C_s$. \\
\indent
Next we shall prove that the optimal matrix satisfies 
\begin{equation}\label{Copt}
  \widetilde C_s=\widetilde C=\frac{d}{\tr(K)}\,I_d=C.
\end{equation}
If the optimal $0\ne\widetilde C_s\in\SMU$ was not proportional to $I_d$, we could ``reduce'' $\widetilde C_s$ by the matrix 
$$
  B:=\widetilde C_s-\lambda\,I_d\ge0\quad
  \mbox{ with } \lambda:=\lambda_{min}(\widetilde C_s),
$$ 
without changing the smallest eigenvalue. 
Moreover $\widetilde C_s-B=\lambda\,I_d$ satisfies
$$
  \tr\left(K^{1/2}[\lambda\,I_d]K^{1/2}\right) =
  \tr(K^{1/2}\widetilde C_sK^{1/2}) - \tr(K^{1/2}BK^{1/2})\le d, 
$$
and hence $\lambda\,I_d$ is another optimal matrix of the above optimization problem. From the equality requirement in \eqref{trD} and $\tr(K^{1/2}\widetilde C_sK^{1/2})=d$ we then conclude $B=0$. Hence $\widetilde C_s$ is proportional to $I_d$, and equality in \eqref{trD} yields
$\lambda_{min}(\widetilde C_s)=\frac{d}{\tr(K)}$, finishing the proof.
\end{proof}
With this lemma we can identify the symmetric FP-equation with steady state $f_{\infty,K}$ that exhibits maximum initial decay as
\begin{equation}
\label{eq:FPopt}
 \partial_{t} f_t = \frac{d}{\tr(K)} \mathrm{div}_{x}{(K\nabla_x f_t+xf_t)}, \qquad x \in \R^d, \ t \in (0,\infty).
\end{equation}
Its initial decay rate, $\frac{d}{\tr(K)}$ is larger then that of \eqref{splitFP}, namely $\frac{1}{\lambda_{max}(K)}$. We recall that the optimal FP-equations constructed in the proof of Theorem  \ref{theo:maintheo}(a) are all hypocoercive, satisfying $\rank(\widetilde D_{opt})=1$, where $\widetilde D_{opt}=(\widetilde C_{opt})_s$. Hence $\lambda_{min}\big((\widetilde C_{opt})_s\big)=0$, and the corresponding propagator norm behaves like $1+\mathcal O(t^2)$, see \eqref{taylor}. Therefore it is obvious that, for small time, the symmetric FP-equations \eqref{splitFP} and \eqref{eq:FPopt} both decrease the FP-propagator norm faster than the hypocoercive FP-evolutions from Theorem  \ref{theo:maintheo}(a). This is illustrated on a 2D example in Figure \ref{fig2}, right.\\
}

\bea{
For the rest of this section we shall base our discussion of using time-de\-pen\-dent coefficients on the concrete example from \S\ref{sec:4.2}, again with $\eps=0.05$, since a general theory of it seems unreachable to us for the moment. In a numerical case study we shall analyze the FP-propagator norm $\|S(t,0)\|_{\mathcal{B}(V_0^{\perp})}$, as a function of time. In the past it would have been quite a challenge to compute (not just to estimate) this norm. But due to Theorem \ref{theo:t-dep} this has become easy for FP-equations with linear drift.
}

\begin{figure}[htbp]
\begin{center}
\includegraphics[scale=0.42]{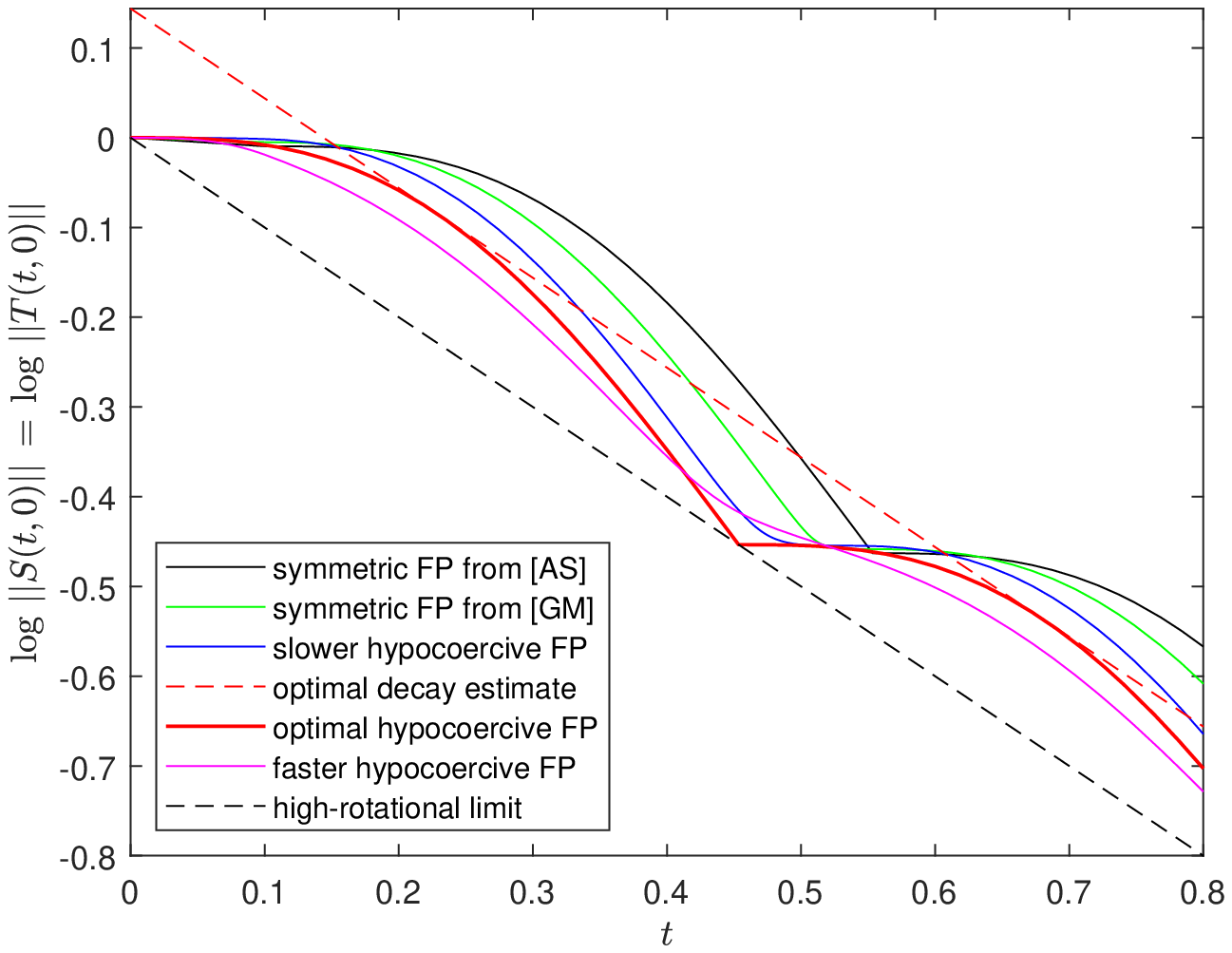}
\hfill
\includegraphics[scale=0.42]{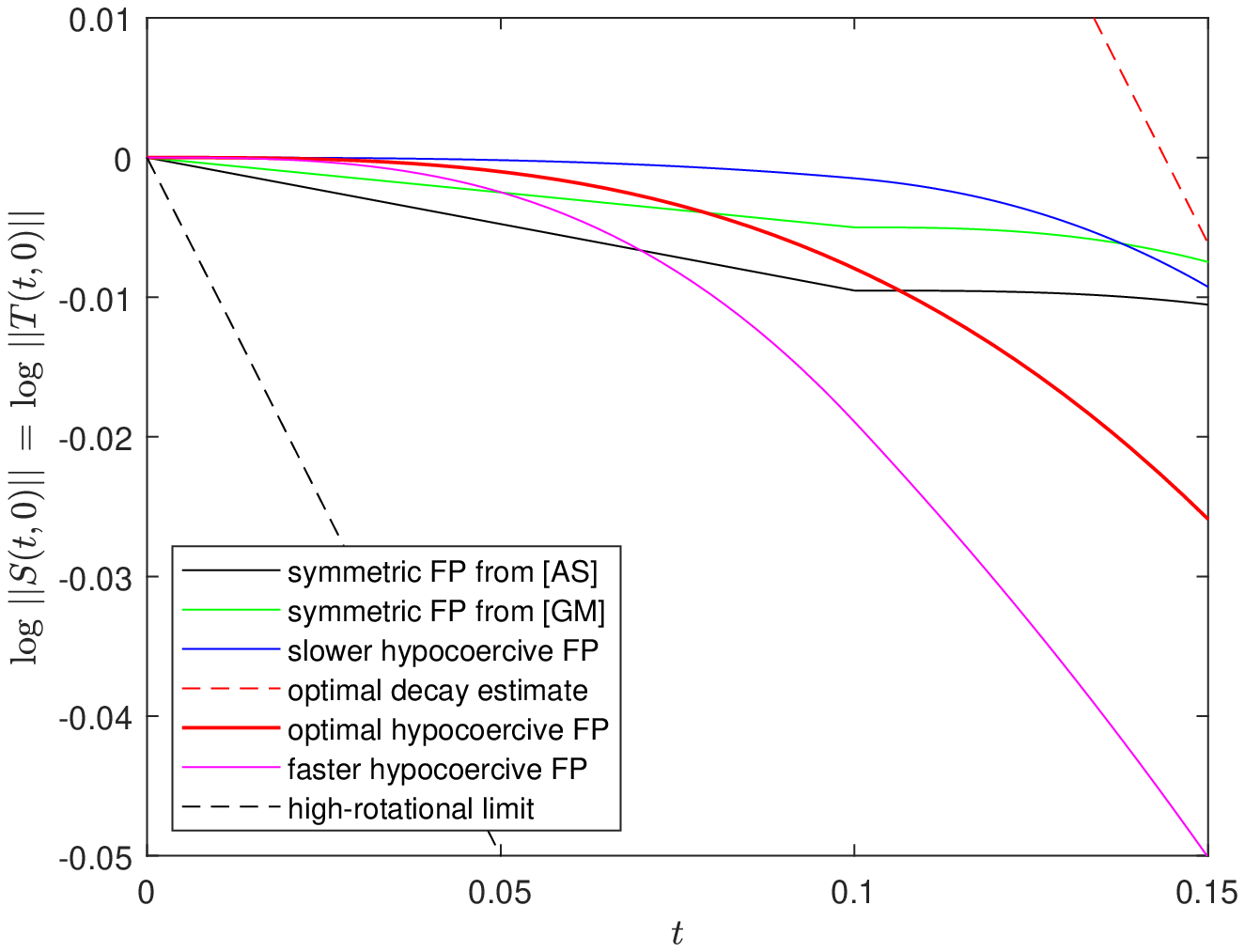}
\end{center}
\caption{\bea{Left: For $c=\sqrt{4/3}$, the FP- and ODE-propagator norms are given for hypocoercive FP-equations with piecewise constant coefficients, using 5 different values on $0\le t\le0.1$: The solid red curve corresponds to the optimal, constant matrices from Theorem \ref{theo:maintheo}(a) as reference case, and the dashed red curve is the corresponding decay estimate \eqref{eq:mainIn}. The initially symmetric FP-equations from \eqref{eq:FPopt} and \eqref{splitFP} are given by the black and green solid curves, respectively. Hypocoercive FP-equations with slower and faster rotational drift are represented, respectively, by the blue and magenta solid curves. \\
Right: a zoom of the plot, close to $t=0$. Colors only online.}  }
\label{fig3}
\end{figure}

\bea{
In Figure \ref{fig3} we shall compare the decay of the FP- and corresponding ODE-propagator norms for 5 cases of FP-equations with piecewise constant coefficient matrices, as in \eqref{splitFP}:
$$
  S(t,0)= \begin{cases}
e^{-L_{C_i,D_i}t} & 0\le t\le t_0 \\
e^{-L_{C_1,D_1}(t-t_0)} e^{-L_{C_i,D_i}t_0} & t> t_0 
\end{cases} ,\quad
  T(t,0)= \begin{cases}
e^{-\widetilde C_i t} & 0\le t\le t_0 \\
e^{-\widetilde C_1 (t-t_0)} e^{-\widetilde C_it_0} & t> t_0 
\end{cases} .
$$
Choosing $t_0=0.1$, we use on the interval $(t_0,\infty)$ always the same matrices, namely those from \eqref{opt-matrices} with $\mu=7$, which is the optimal hypocoercive FP-evolution from Theorem \ref{theo:maintheo} for the multiplicative constant $c=\sqrt{4/3}$. For the interval $[0,t_0]$ we shall compare the following cases:
\begin{enumerate}[(FP1)]
\item \label{FP1} This reference case uses the same coefficients as for $t>t_0$, i.e.:
\begin{equation}\label{ref-matrices}
  D_1=\diag(0,2),\quad C_1= \begin{pmatrix}
0 & -\frac{7}{\sqrt{\eps}}
\\
7\sqrt{\eps}  & 2 
\end{pmatrix}, \quad 
\widetilde{C}_1=\begin{pmatrix}
0 & -7
\\
 7 & 2
\end{pmatrix}.
\end{equation}
Figure \ref{fig3} also includes the sharp upper and lower envelops of the resulting propagator norm (as function of $t$).
\item \label{FP2} The symmetric FP-equation from \cite{GM}, and shown in \eqref{splitFP} is determined by the matrices
$$
  D_2=I_2,\quad C_2=\widetilde{C}_2= \diag(\eps, 1).
$$
\item \label{FP3} The symmetric FP-equation \eqref{eq:FPopt} with maximum initial decay is determined by the matrices
$$
  D_3=\frac{2\eps}{1+\eps}\diag(\frac{1}{\eps},1),\quad C_3=\widetilde{C}_3=\frac{2\eps}{1+\eps} I_2.
$$
\item \label{FP4} A hypocoercive FP-equation with slower rotational part than in \eqref{ref-matrices} is determined by the matrices
$$
  D_4=\diag(0,2),\quad C_4= \begin{pmatrix}
0 & -\frac{3}{\sqrt{\eps}}
\\
3\sqrt{\eps}  & 2 
\end{pmatrix}, \quad 
\widetilde{C}_4=\begin{pmatrix}
0 & -3
\\
 3 & 2
\end{pmatrix}.
$$
\item \label{FP5} A hypocoercive FP-equation with faster rotational part than in \eqref{ref-matrices} is determined by the matrices
$$
  D_5=\diag(0,2),\quad C_5= \begin{pmatrix}
0 & -\frac{11}{\sqrt{\eps}}
\\
11\sqrt{\eps}  & 2 
\end{pmatrix}, \quad 
\widetilde{C}_5=\begin{pmatrix}
0 & -11
\\
 11 & 2
\end{pmatrix}.
$$
Note that (FP\ref{FP4}) and (FP\ref{FP5}) are both of the form \eqref{opt-matrices}.
\end{enumerate}
}

\bea{
First we need to fix the criterion for comparing these 5 FP-equations with split coefficients. As one sees from Figure \ref{fig3}, adapting the FP-equation only on the initial time interval $[0,t_0]$ has a highly nonlocal-in-$t$ effect. Hence, it does not make sense to compare the norm-curves pointwise in time. Following the paradigm of \S\ref{sec3}, it is appropriate to compare again the corresponding best exponential decay estimates \eqref{eq:mainIn}. Since all compared FP-equations coincide for large time, or more precisely on $(t_0,\infty)$, their exponential decay rate is the same, and it suffices to compare the multiplicative constant of the (sharp) decay estimates. \\
\indent
Now we shall replace in the reference FP-equation (FP\ref{FP1}) the initial phase by a symmetric evolution:  With both options (FP\ref{FP2}) and (FP\ref{FP3}) the propagator norm decays initially faster than for the reference FP-equation (see Figure \ref{fig3}, right), but this backfires at later times: In both cases the upper envelop for the whole norm-function on $[0,\infty)$ and hence the multiplicative constant $c$ is larger than for the reference case (FP\ref{FP1}) (see Figure \ref{fig3}, left).\\
\indent
Finally we shall replace in the reference FP-equation (FP\ref{FP1}) the initial phase by a hypocoercive evolution having an anti-symmetric part of $\widetilde C$ that differs from case (FP\ref{FP1}). With the slower rotational part in case (FP\ref{FP4}) the multiplicative constant $c$ is increased (see Figure \ref{fig3}, left), but when using initially the faster rotational part from case (FP\ref{FP5}), the multiplicative constant $c$ is decreased!}

\begin{figure}[htbp]
\begin{center}
\includegraphics[scale=0.7]{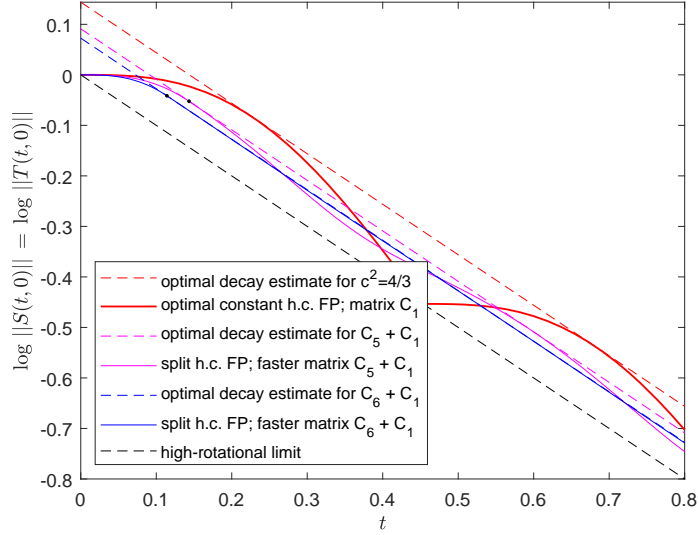}
\end{center}
\caption{\bea{For $c=\sqrt{4/3}$, the FP- and ODE-propagator norms are given for hypocoercive (h.c.) FP-equations with piecewise constant coefficients, using 3 different values on $0\le t\le t_0$: The solid red curve corresponds to the optimal, constant matrices from Theorem \ref{theo:maintheo}(a) as reference case. Hypocoercive FP-equations with the faster rotational drift matrices (FP\ref{FP5}), (FP6) 
are represented by the magenta and blue solid curves, respectively. The dashed curves are the corresponding decay estimate \eqref{eq:mainIn}. The discontinuity points $t_0$ of the coefficient matrices are marked with black dots.}  }
\label{fig4}
\end{figure}

\bea{
While we present in Figure \ref{fig3} the plots only for $t_0=0.1$, the results for other values of $t_0>0$ are qualitatively the same.
Choosing $t_0\approx0.1434$ (i.e.\ the first point of tangency between $\|e^{-\widetilde C_5t}\|_{\mathcal{B}(\R^d)}$ and its sharp exponential decay estimate $\sqrt{6/5}e^{-t}$, see Figure \ref{fig4}
) in the split case (FP\ref{FP5}) reduces the multiplicative constant to $c=\sqrt{6/5}$. Note that this is also the sharp constant for the {\it non-split} FP-equation involving the matrices $(C_5,\,D_5)$. This means that the same decay quality (in the above defined sense) can be obtained with the constant coefficient matrices $(C_5,\,D_5)$ for all time or just a short initial layer with $(C_5,\,D_5)$ and then evolving with $(C_1,\,D_1)$ for $t>t_0$. 
The multiplicative constant can be reduced even further, e.g.\ with the following choice of matrices on the interval $[0,0.11413]$ (see Figure \ref{fig4}):
\begin{enumerate}[(FP6)]
\item \label{FP1a} $$
  D_6=\diag(0,2),\quad C_6= \begin{pmatrix}
0 & -\frac{13.8}{\sqrt{\eps}}
\\
13.8\sqrt{\eps}  & 2 
\end{pmatrix}, \quad 
\widetilde{C}_6=\begin{pmatrix}
0 & -13.8
\\
 13.8 & 2
\end{pmatrix}
$$
\end{enumerate}
This example of time-dependent FP-coefficients is also algorithmically relevant, since $\|C_1\|_\F < \|C_6\|_\F$. Hence, longer time steps could be used in a discretization of the split FP-equation for $t>t_0$.
}


\section{Conclusion}\label{sec:5}
For any given anisotropic Gaussian steady state \eqref{eq:finf} with covariance matrix $K$, we analyzed the construction of non-symmetric FP-equations \eqref{eq:FP} that show fastest decay towards the unique normalized steady state $f_{\infty,K}$. Building upon preceding results (in particular \cite{LNP, GM}) we proved that optimal exponential decay with small multiplicative constants (as in \eqref{eq:hypowithJ}, and uniformly in $f_0$) can be achieved with a single FP-equation, without having to split off an initial evolution phase. Thereby, the maximum decay rate $\lambda_{opt}=\max(\sigma(K^{-1}))$, and the infimum of the multiplicative constants $c_{inf}=1$. By contrast, the best multiplicative constant obtainable in \cite{GM} was bounded below by $\sqrt{\kappa(K)\,e}$. Hence, the gain provided here for the multiplicative constant is particularly important when $\kappa(K)$ is large, i.e.\ when the original, symmetric FP-dynamics includes very different time scales due to very different eigenvalues in $K$.

More precisely, for any given multiplicative constant $c>1$ we were able to construct explicitly a non-symmetric FP-equation of form \eqref{eq:FP} with constant drift matrix $C_{opt}(c)$ and diffusion matrix $D_{opt}(c)$ such that the exponential decay estimate \eqref{eq:hypowithJ} holds with the parameters $(\lambda_{opt},\,c)$. For given $c$ and variable space dimension $d$, we were able to reduce the growth estimate on these drift matrices to $\mathcal O(d^{3/2})$, down from $\mathcal O(d^2)$ given in \cite{GM}.

In explicit 2D examples we illustrated, \bea{both analytically and numerically,} that the infimum of the multiplicative constant, $c_{inf}=1$ corresponds to the limit of adding a highly rotational, non-symmetric drift to the original FP-equation.

\bea{To round off our analysis we presented a numerical case study on a FP-equa\-tion in 2D with piecewise constant coefficient matrices. This showed two unexpected phenomena: 
First, no symmetric FP-evolution on an initial time layer was able to improve the overall decay behavior; in fact it always got worse than in the time-independent case.
Second, replacing on an initial time layer the non-sym\-met\-ric drift by a higher rotational one (and then returning to the original drift for all time) can reduce the multiplicative constant for the whole evolution to a level that pertains to a ``larger'' drift matrix $C$.
}

\bea{
\appendix
\section{Proof of Theorem \ref{theo:t-dep}}
\begin{proof}[Proof-idea]
First, the coordinate transformation $\tilde x:=K^{-1/2}x$ and $\tilde f(\tilde x):= (\det K)^{1/2}\, f(K^{1/2}\tilde x)$ transforms \eqref{eq:tFP} into the {\it normalized FP-equation}
\begin{equation}\label{eq:ntFP}
 \partial_{t} \tilde f_t = -\widetilde L(t)\tilde f_t := \mathrm{div}_{\tilde x}{(\widetilde D(t)\nabla_{\tilde x} \tilde f_t+\widetilde C(t)\tilde x\tilde f_t)}, \qquad \tilde x \in \R^d, \ t \in (0,\infty),
\end{equation}
where $\widetilde D(t):=K^{-1/2}D(t)K^{-1/2}$. This FP-equation is naturally considered in $\widetilde \H:=L^2(\R^d,\tilde f_\infty^{-1})$, and the (transformed) steady state is
$$
  \tilde f_\infty(\tilde x)=(2\pi)^{-d/2}\,e^{-|\tilde x|^2/2}.
$$
This transformation preserves the norm of the solution: $\|f_t\|_{\H}=\|\tilde f_t\|_{\widetilde \H},\:t\ge0$. Hence the propagator norms of \eqref{eq:tFP} and \eqref{eq:ntFP} coincide: $\left\|S(t_2,t_1)\right\|_{\mathcal{B}(V_0^{\perp})} =
\left\|\widetilde S(t_2,t_1)\right\|_{\mathcal{B}(\widetilde V_0^{\perp})}$.\\
\indent
Next, one decomposes $\widetilde\H$ into mutually orthogonal subspaces $\widetilde V^{(m)},\:m\in\N_0$, which are each invariant under the operators $\widetilde L(t)$ $\forall t\ge0$:
$$
  \widetilde \H=\bigoplus_{m\in\N_0}{\!\!\!}^\perp \,\widetilde V^{(m)},
$$
with
$$
  \widetilde V^{(m)}:=\span\{g_\alpha(\tilde x):=(-1)^{|\alpha|}\nabla^\alpha \tilde f_\infty(\tilde x)\,:\, \alpha\in\N_0^d,\,|\alpha|=m\}.
$$
Decomposing the solution of \eqref{eq:ntFP} into these subspaces as
$$
  \tilde f_t(\tilde x)=\sum_{\alpha\in\N_0^d} \tilde d_\alpha(t)\,\frac{g_\alpha(\tilde x)}{\|g_\alpha\|_{\widetilde\H}},
$$
yields the estimates
$$
  \sum_{|\alpha|=m} |\tilde d_\alpha(t_2)|^2 \le h(t_2,t_1)^{2m} 
  \left( \sum_{|\alpha|=m} |\tilde d_\alpha(t_1)|^2 \right),\quad
  0\le t_1\le t_2<\infty,\:m\in\N,
$$
with
$$
  h(t_2,t_1) := \|T(t_2,t_1)\|_{\mathcal{B}(\R^d)} \le 1,\quad
  0\le t_1\le t_2<\infty.
$$
On the one hand this shows that
$$
  \|\tilde f_{t_2}-\tilde f_\infty\|_{\widetilde \H}
  =\|\widetilde S(t_2,t_1)(\tilde f_{t_1}-\tilde f_\infty)\|_{\widetilde \H} \le \|T(t_2,t_1)\|_{\mathcal{B}(\R^d)} \|\tilde f_{t_1}-\tilde f_\infty\|_{\widetilde \H},\quad
  0\le t_1\le t_2<\infty.
$$
On the other hand we can use initial conditions $\tilde f_{t_1}\in \widetilde V^{(1)}$, noting as in \cite[\S4.2]{ASS} that the coefficient vector $\tilde d^{(1)}(t):=\left(\tilde d_\alpha(t)\right)_{|\alpha|=1}\in\R^d$ evolves according to
$$
  \frac{d}{dt} \tilde d^{(1)} = -\widetilde C(t)\,\tilde d^{(1)},
$$
i.e.\ the drift ODE of the FP-equation.
This implies the reverse inequality
$$
  \left\|\widetilde S(t_2,t_1)\right\|_{\mathcal{B}(\widetilde V_0^{\perp})} \ge \left\|T(t_2,t_1)\right\|_{\mathcal{B}(\R^d)}, \qquad \forall 0\le t_1\le t_2<\infty,
$$
and hence the equality \eqref{prop-equal-t} follows.
\end{proof}
}


\section*{Acknowledgement}
The authors were partially supported by the FWF (Austrian Science Fund) funded SFB \#F65 and the FWF-doctoral school W 1245. We acknowledge fruitful discussions with Jos\'e A. Carrillo, who originally proposed this topic to us. We are also grateful to the anonymous referees, whose suggestions helped to improve this work.

\end{document}